\RequirePackage{fix-cm}
\RequirePackage{amsmath}
\pdfoutput=1
\documentclass[10pt]{amsart}
\usepackage[margin=1in]{geometry}
\usepackage[utf8]{inputenc}

\usepackage{url}
\usepackage{multirow}
\usepackage[foot]{amsaddr}

\usepackage{xfrac}
\usepackage[shortlabels]{enumitem}
\newlist{algolist}{enumerate}{3}
\setlist[algolist,1]{label=(\arabic*), ref=(\arabic*)}
\setlist[algolist,2]{label=(\alph*), ref=(\arabic{algolisti})-(\alph*)}
\setlist[algolist,3]
{label=(\roman*), ref=(\arabic{algolisti})-(\alph{algolistii})-(\roman*)}
\usepackage{color}
\usepackage{datetime}
\usepackage{soul}
\usepackage{verbatim}

\usepackage{adforn}
  
\usepackage{graphicx}
\usepackage[percent]{overpic}
\usepackage[caption=false]{subfig}
\usepackage{float}
\newfloat{algorithm}{t}{lop}

\allowdisplaybreaks

\usepackage{amssymb}              
\usepackage{amsfonts}
\usepackage{bm}                   
\usepackage{bbm}                  
\usepackage{array}                
\usepackage{mathtools}            
\numberwithin{equation}{section}  

\usepackage[bookmarks=true, hidelinks]{hyperref}
\usepackage{amsthm}               
\theoremstyle{plain}
\newtheorem{theorem}{Theorem}[section]

\newtheorem{corollary}[theorem]{Corollary}
\theoremstyle{definition}
\newtheorem{definition}[theorem]{Definition}

\theoremstyle{remark}

\usepackage[capitalize,noabbrev]{cleveref}   
\makeatletter
\if@cref@capitalise
\crefname{algolisti}{Step}{Steps}
\crefname{algolistii}{Step}{Steps}
\crefname{algolistiii}{Step}{Steps}
\else
\crefname{algolisti}{step}{steps}
\crefname{algolistii}{step}{steps}
\crefname{algolistiii}{step}{steps}
\fi
\makeatother

\DeclarePairedDelimiter\abs{\lvert}{\rvert} 
\DeclareMathOperator*{\arginf}{arg\,inf} 

\providecommand{\N}{\mathbb{N}} 
\providecommand{\R}{\mathbb{R}} 
\DeclarePairedDelimiter\set{\{}{\}}
\DeclarePairedDelimiterX\setc[2]{\{}{\}}{\,#1 \;\delimsize\vert\; #2\,}
\providecommand{\vc}[1]{\mathbf{#1}}

\begin{document}
\title{Uniqueness of optimal solutions for semi-discrete transport with 
$\bm{p}$-norm cost functions}
\author{J.D.\ Walsh III$^\dagger$}
\address{$^\dagger$School of Mathematics \\ Georgia Techn \\
Atlanta, GA 30332 U.S.A.\\Tel.: +1 404-894-4401 Fax: +1 404-894-4409}
\email{jdwalsh03@gatech.edu}
\thanks{{\bf Acknowledgments}. 
This material is based upon work supported by the National Science Foundation 
Graduate Research Fellowship Program under Grant No. DGE-1650044. Any 
opinions, findings, and conclusions or recommendations expressed in this 
material are those of the authors and do not necessarily reflect the views of 
the National Science Foundation.}



\begin{abstract}
Semi-discrete transport can be characterized in terms of real-valued shifts. 
Often, but not always, the solution to the shift-characterized problem 
partitions the continuous region. This paper gives examples of 
when partitioning fails, and offers a large class of semi-discrete transport 
problems where the shift-characterized solution is always a partition.
\end{abstract}

\maketitle

\pagestyle{myheadings}
\thispagestyle{plain}
\markboth{J.D.\ Walsh III}{Uniqueness of optimal solutions for semi-discrete 
transport}

\section{Introduction}
Optimal transport offers a way to measure the distance between two probability 
spaces, $X$ and $Y$.
In the class of transport problems known as semi-discrete optimal transport, 
the probability distribution on $X$ is almost-everywhere continuous and the 
probability distribution on $Y$ is discrete, with $N$ points of positive 
measure.
Given minimal assumptions, described below, the semi-discrete problem always 
has at least one solution that partitions $X$ into $N$ regions based on 
transport destination.

R{\"u}schendorf and Uckelmann developed a way to characterize semi-discrete 
transport in terms of a set of real-valued shifts.
This shift characterization often results in a solution that partitions $X$ 
into $N$ regions.
Unfortunately, the shift characterization does not always partition $X$.
This important fact has not always been recognized or clearly expressed in the 
literature; see~\cite{Ruschendorf2007a,Ruschendorf1997a,Ruschendorf2000a}.
To remedy that ambiguity, this paper gives clear, specific examples where 
shift-characterized partitioning fails, and it offers a large class of problems 
where the shift characterization is guaranteed to partition $X$.

\section{Background}
\subsection{General optimal transport: the Monge-Kantorovich and Monge 
problems}\label{s:transport_problem}
Though this paper focuses on the semi-discrete problem, it is worth
describing it in terms of the more general, Monge-Kantorovich transport problem.

\begin{definition}[Monge-Kantorovich problem]\label{MKproblem}
Let $X,\,Y \subseteq \R^d$, let $\mu$ and $\nu$ be probability densities 
defined on $X$ and $Y$, and let $c(\vc{x},\,\vc{y}) : X \times Y \to \R$ be a 
continuous measurable 
\emph{ground cost} function.
Define the set of \emph{transport plans}
\begin{equation}\label{MK-1}
\Pi(\mu,\,\nu) := \set*{\pi \in \mathcal{P}(X \times Y) \left|
\begin{array}{c}
\pi[A \times Y] = \mu[A],\,
 \pi[X \times B] = \nu[B] \ ,\\
 \forall \text{ meas.\ }
 A \subseteq X,\, B \subseteq Y
\end{array}
\right. },
\end{equation}
where $\mathcal{P}(X \times Y)$ is the set of probability measures on the 
product space,
and define the \emph{primal cost} function $P: \Pi(\mu,\,\nu) \to \R$ as
\begin{equation}\label{primal_cost}
P(\pi) := \int_{X \times Y} c(\vc{x},\,\vc{y})\, d\pi(\vc{x},\,\vc{y}).
\end{equation}
The Monge-Kantorovich problem is to
find the \emph{optimal primal cost}
\begin{equation}\label{MK-2}
P^* := \inf_{\pi \in \Pi(\mu,\,\nu)}\, P(\pi),
\end{equation}
and an associated \emph{optimal transport plan}
\begin{equation}\label{MK-pi-star}
\pi^* := \arginf_{\pi \in \Pi(\mu,\,\nu)}\, P(\pi).
\end{equation}
\end{definition}

Under the conditions given, an optimal transport plan, $\pi^*$, is guaranteed 
to exist.
However, $\pi^*$ may not be unique, or even 
a.e.-unique.
Furthermore, the existence of $\pi^*$, an optimal \emph{plan}, does not ensure 
that $\pi^*$ is a \emph{map}, or that an optimal map exists.
Nonetheless, consider the form such an optimal map would take.

\begin{definition}[Monge problem]\label{Mproblem}
In certain cases, there exists at least one solution to the semi-discrete 
Monge-Kantorovich problem that does not split transported masses.
In other words, there exists some $\pi^*$ such that
\begin{equation}\label{Tstar}
\pi^*(\vc{x},\,\vc{y}) = \pi^*_{\scriptscriptstyle{T^*}}(\vc{x},\,\vc{y}) :=
\mu(\vc{x})\,\chi[\vc{y} = T^*(\vc{x})],
\end{equation}
where $T^* : X \to Y$ is a measurable map called the \emph{optimal
transport map}.
When such a $\pi^*$ exists, we say the solution also solves the Monge 
problem.

If the Monge problem has a solution, we can assume without loss of 
generality that every $\pi \in \Pi(\mu,\,\nu)$ satisfies
\begin{equation}\label{Tplan}
\pi(\vc{x},\,\vc{y}) = \pi_{\scriptscriptstyle{T}}(\vc{x},\,\vc{y}) :=
\mu(\vc{x})\,\chi[\vc{y} = T(\vc{x})],
\end{equation}
for some measurable transport map $T : X \to Y$, and that the primal cost can 
be written
\begin{equation}\label{Mprimal_cost}
P(\pi) := \int_{X} c(\vc{x},\,T(\vc{x}))\, d\mu(\vc{x}).
\end{equation}
\end{definition}

\subsection{Semi-discrete optimal transport and the 
shift characterization}\label{s:semi-discrete}
The semi-discrete optimal transport problem is the Monge-Kantorovich
problem of \cref{MKproblem}, with restrictions on $\mu$ and $\nu$:
\begin{enumerate}
\item
Assume that $\mu$ satisfies the following:
\begin{enumerate}
\item
$\mu$ is bounded.
\item
$\mu$ is nonatomic.
\item
$\mu$ is continuous except on a set of Lebesgue measure zero.
\item
The support of $\mu$ is contained in the convex compact region $A \subseteq X$.
\end{enumerate}
\item
Assume $\nu$ has exactly $n \geq 2$ non-zero values, located at 
$\set{\vc{y}_i}_{i=1}^n 
\subseteq Y$.
\end{enumerate}
Because $c$ is continuous and $\mu$ is nonatomic, at least 
one solution to the semi-discrete Monge-Kantorovich problem also satisfies 
the Monge problem, described in \cref{Mproblem}; see~\cite{Rachev1998a}.
Thus, by applying \cref{Tplan}, we can assume without loss of generality that 
any transport plan $\pi$ has an associated map $T$, and that $T$ partitions $A$ 
into $n$ sets $A_i$, 
where $A_i$ is the set of points in $A$ that are transported by $T$ to 
$\vc{y}_i$.
Using this partitioning scheme in combination with \cref{Mprimal_cost} allows 
us to rewrite the primal cost function for the semi-discrete problem as
\begin{equation}\label{primal_cost_sd}
P(\pi) := \sum_{i=1}^n \int_{A_i} c(\vc{x},\,\vc{y}_i)\, 
d\mu(\vc{x}).
\end{equation}

This idea of sets $A_i$ is central to describing the shift 
characterization of the semi-discrete optimal transport problem.
The following definition is based on one given by 
R\"{u}schendorf and Uckelmann in~\cite{Ruschendorf2007a,Ruschendorf2000a}.

\begin{definition}[Shift characterization]\label{ShiftChar}
Let $\set{a_i}_{i=1}^n$ be a set of $n$ finite values, referred 
to as \emph{shifts}.
Define
\begin{equation}\label{Fdef}
F(\vc{x}) := \max_{1\leq i\leq n} \set{a_i - c(\vc{x},\,\vc{y}_i)}.
\end{equation}
For $i \in \N_n$, where $\N_n = \set{1,\,\ldots,\,n}$, let
\begin{equation}\label{Aidef}
A_i  := \set{\vc{x} \in A \mid F(\vc{x}) = a_i - c(\vc{x},\,\vc{y}_i)}.
\end{equation}
Note that $\cup_{i=1}^nA_i=A$.
The problem of determining an optimal transport plan $\pi^*$ is equivalent to 
determining shifts $\set{a_i}_{i=1}^n$ such that for all $i \in 
\N_n$, the total mass transported from $A_i$ to $\vc{y}_i$ equals 
$\nu(\vc{y}_i)$.
\end{definition}

\subsection{Formalizing the shift-characterized partition}
``Partitioning'' $A$ is described in~\cite{Ruschendorf2000a,Ruschendorf2007a} 
as $\mu(A_i) = \nu(\vc{y}_i)$. However, it is beneficial to describe the 
shift-characterized partition in more detail.
Doing so requires a few additional definitions.

\begin{definition}[Boundaries and boundary sets]
For all $i,\,j \in \N_n$ such that $i \neq j$, let
\begin{equation}\label{Aij}
A_{ij} := A_i \cap A_j.
\end{equation}
The \emph{boundary set} is defined as
\begin{equation}\label{bdryset}
B := \bigcup_{1 \leq i < n}\, \bigcup_{i < j \leq n} A_{ij}.
\end{equation}
\end{definition}
For all $i,\,j \in \N_n$ such that $i \neq j$, define
$g_{ij} : X \to \R$ as
\begin{equation}\label{gij}
g_{ij}(\vc{x}) := c(\vc{x},\,\vc{y}_i) - c(\vc{x},\,\vc{y}_j).
\end{equation}

\begin{definition}[$F$ $\mu$-partitions $A$]\label{d:part_A}
Let $F$ be as defined in \cref{Fdef}, and the sets $A_i$ as defined in 
\cref{Aidef} for $i \in \N_n$. Then one says $F$ \emph{$\mu$-partitions} the 
set $A$, or $F$ is called a \emph{$\mu$-partition}, if
\begin{enumerate}
\item $\mu(A) < \infty$,
\item for all $i,\,j \in \N_n$, $i\neq j$, $\mu(A_{ij})=0$,
\item
$\sum_{i=1}^n \mu(A_i) = \mu(A)$, and
\item for all $i \in \N_n$, $\mu(A_i) = \nu(\vc{y}_i) > 0$.
\end{enumerate}
\end{definition}

\begin{definition}[Monge under the shift characterization]
We say a transport plan $\pi$ is \emph{Monge under the shift 
characterization} if $\pi$ has an associated transport map $T$, 
a function $F$, as described in \cref{Fdef}, 
and sets $\set{A_i}_{i=1}^n$, as described in \cref{Aidef}, such that for all 
$\vc{x} \in A$,
\begin{equation}
\vc{x} \in \mathring{A}_i \text{ for some } i \in \N_n
\quad\quad \implies \quad\quad
T(\vc{x}) = \vc{y}_i.
\end{equation}
In other words,
$F$ $\mu$-partitions $A$ and
$T$ agrees with $F$ on $A \setminus B$.
\end{definition}

If $\mu(B) > 0$ for the shifts $\set{a_i}_{i=1}^n$, no such transport plan 
$\pi$ can exist, and the transport 
problem itself can be said to be \emph{not Monge under the shift 
characterization}.
Conversely, if $\mu(B) = 0$, then such a transport plan exists, and so the 
transport problem itself is said to be Monge under the shift characterization.
In other words, $F$ $\mu$-partitions $A$ if and only if the transport problem 
is Monge under the shift characterization.

The following result, from~\cite{Dieci2017a}, allows us to go further:

\begin{theorem}\label{t:muPartition}
Suppose one has a semi-discrete transport problem, as described in 
\cref{s:semi-discrete}.
Let $F$ be as defined in \cref{Fdef}, and the sets $A_i$ as defined in 
\cref{Aidef} for $i \in \N_n$. Then $F$ $\mu$-partitions $A$ if and only if 
$\mu(B) = 0$.
\end{theorem}

Taken together, these statements provide a formal definition and 
condition for what it means for the shift-characterized solution to partition 
$A$:

\medskip
\begin{center}
\fbox{
  \parbox[c]{4.75in}{
The shift-characterized semi-discrete transport problem partitions $A$
--- that is, $F$ $\mu$-partitions $A$ --- if and only if the semi-discrete 
transport problem is Monge under the shift characterization, which is true if 
and only if $\mu(B)=0$.
  }
}
\end{center}

\subsection{Uniqueness of semi-discrete transport 
solutions}\label{s:unique_cond}
Given the semi-discrete transport problem described in \cref{s:semi-discrete},
Corollary 4 of \cite{Cuesta1993a} provides a sufficient condition for 
the existence of a Monge solution that is unique $\mu$-a.e.:
\begin{equation}\label{e:suffUniqueMonge}
\mu\left( \setc{\vc{x} \in A}
{c(\vc{x},\,\vc{y}_i) - c(\vc{x},\,\vc{y}_j) = k } \right) = 0
\quad\quad
\forall\, i,\,j \in \N_n,\,i\neq j,
\quad\quad
\forall\, k \in \R.
\end{equation}
If \cref{e:suffUniqueMonge} is satisfied, $\mu(B)=0$. Therefore, if 
\cref{e:suffUniqueMonge}, then the transport problem is Monge under the 
shift characterization and the transport solution is unique $\mu$-a.e.

However, if a transport problem is Monge under the shift 
characterization, then it has a unique $\mu$-a.e.\ shift-characterized 
solution, whether or not \cref{e:suffUniqueMonge} is satisfied.
This statement is formalized and proved in~\cite{Dieci2017a} as the following 
theorem:

\begin{theorem}[The optimal transport map is unique $\mu$-a.e.]\label{UniqueT}
Given a semi-discrete transport problem, let $\pi^*$ and $\tilde{\pi}^*$ be
optimal transport plans that are both Monge under the shift characterization.
If $T$ is a transport map associated with $\pi^*$, and $\widetilde{T}$ a 
transport map associated with $\tilde{\pi}^*$, then $T = \widetilde{T}$ except 
on a set of $\mu$-measure zero.
\end{theorem}

\section{Mathematical support}\label{s:math}
While \cref{e:suffUniqueMonge} implies $\mu(B)=0$, the converse is not true,
as \cref{s:1andInfNorm} shows. Next,
\cref{s:OtherPnorms} identifies a large class of problems where both 
conditions hold and the solution is always unique $\mu$-a.e.

\subsection{Partitioning with the 1-norm 
and \texorpdfstring{$\bm{\infty}$}{∞}-norm}\label{s:1andInfNorm}
Let $X = [0,\,1]^2$, $Y = \set{\vc{y}_1,\,\vc{y}_2}$, and
let $\mu$ be the continuous uniform distribution.
This simple setup can be used to demonstrate
failure to partition
for both the 
uniform norm ($\infty$-norm) and the Manhattan norm ($1$-norm).

\subsubsection{The uniform norm}
Let $\vc{y}_1 = (\sfrac{1}{4},\,\sfrac{1}{2})$ and $\vc{y}_2 = 
(\sfrac{3}{4},\,\sfrac{1}{2})$, and let
$c : \R^2 \times \R^2 \to \R$ be the 
uniform norm ($\infty$-norm):
$c(\vc{x},\,\vc{y}) = \max_{i \in \set{1,\,2}}\, \abs{x_i - y_i}$
for all $\vc{x} = (x_1,\,x_2) \in X$,
$\vc{y} = (y_1,\,y_2) \in Y$.
Consider two examples:
\begin{enumerate}
\item
If $\nu(\vc{y}_1) = \sfrac{1}{32}$, then $\nu(\vc{y}_2) = \sfrac{31}{32}$.
In this case, $\mu(B) = \sfrac{1}{16}$,
and the shift-characterized solution fails to partition $A$.
See \cref{f:infNorm1-32}.
\item
However, if $\nu(\vc{y}_1) = \sfrac{1}{8}$, then $\nu(\vc{y}_2) = 
\sfrac{7}{8}$, 
$\mu(B) = 0$ and the shift-characterized solution does partition $A$.
See \cref{f:infNorm1-8}.
\end{enumerate}
Even though one of the problems illustrated in \cref{f:infNorm} results in a 
partition, \cref{e:suffUniqueMonge} fails in both cases:
\begin{equation*}
\mu\left( \setc*{\vc{x} \in A}
{c(\vc{x},\,\vc{y}_2) - c(\vc{x},\,\vc{y}_1) = k } \right) = 
\tfrac{1}{16}
\quad\quad
\text{ if }\quad\quad
k \in \set*{-\tfrac{1}{2},\,\tfrac{1}{2}}.
\end{equation*}
In general, for this choice of $X$, $Y$, $\mu$, and $c$, the 
shift-characterized solution partitions $A$ if and only if
\begin{equation*}
\nu(\vc{y}_1) \in \left(\tfrac{1}{16},\,\tfrac{15}{16}\right).
\end{equation*}
Thus, when $c$ is the uniform norm, one can have $\mu(B)=0$, giving a 
shift-characterized partition of $A$ that is unique $\mu$-a.e., whether or not 
\cref{e:suffUniqueMonge} is satsified. 

\begin{figure}[htpb]
  \centering
  \subfloat[$\nu(\vc{y}_1) = 1/32$]{%
    \label{f:infNorm1-32}
    \centering
    \resizebox*{0.3\textwidth}{!}{%
    \begin{overpic}[width=\textwidth]{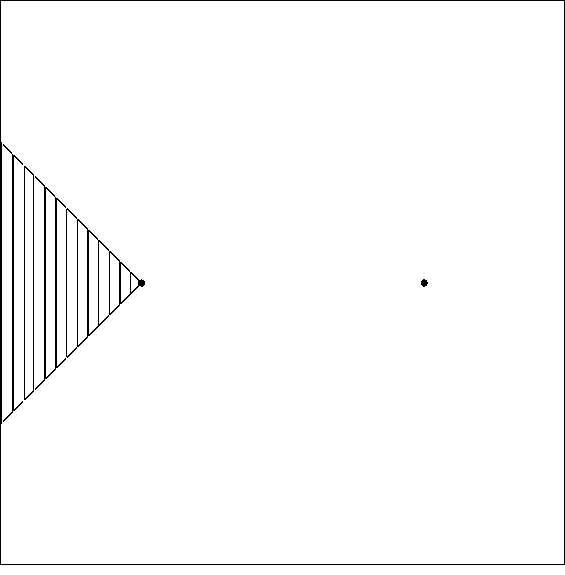}
    \put (27.5,50.0) {\scalebox{2.0}{$\vc{y}_0$}}
    \put (72.5,52.5) {\scalebox{2.0}{$\vc{y}_1$}}
    \end{overpic}}
  }%
  \quad\quad\quad\quad%
  \subfloat[$\nu(\vc{y}_1) = 1/8$]{%
    \label{f:infNorm1-8}
    \centering
    \resizebox{0.3\textwidth}{!}{%
    \begin{overpic}[width=\textwidth]{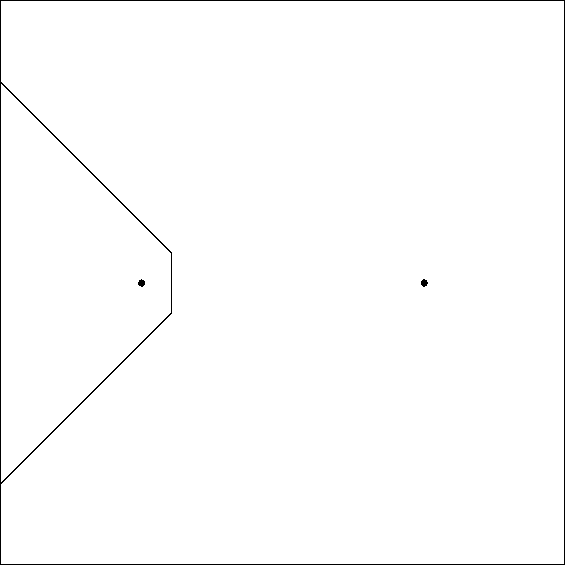}
    \put (22.5,52.5) {\scalebox{2.0}{$\vc{y}_0$}}
    \put (72.5,52.5) {\scalebox{2.0}{$\vc{y}_1$}}
    \end{overpic}}
  }
\caption{$\infty$-norm partitioning example}\label{f:infNorm}
\end{figure}

\subsubsection{The Manhattan norm}
Now let $\vc{y}_1 = (\sfrac{1}{4},\,\sfrac{1}{4})$ and $\vc{y}_2 = 
(\sfrac{3}{4},\,\sfrac{3}{4})$. Let
$c : \R^2 \times \R^2 \to \R$ be the Manhattan norm (1-norm):
$c(\vc{x},\,\vc{y}) = \abs{x_1 - y_1} + \abs{x_2 - y_2}$
for all
$\vc{x} = (x_1,\,x_2) \in X$,
$\vc{y} = (y_1,\,y_2) \in Y$.
Consider three examples:
\begin{enumerate}
\item\label{en:std_auren}
If $\nu(\vc{y}_1) = \sfrac{1}{2}$, then $\nu(\vc{y}_2) = \sfrac{1}{2}$.
In this case, $\mu(B) = \sfrac{1}{8}$,
and the shift-characterized solution fails to partition $A$.
See \cref{f:auren1-2} for an illustration.
\item
If $\nu(\vc{y}_1) = \sfrac{1}{32}$, then $\nu(\vc{y}_2) = \sfrac{31}{32}$ and 
$\mu(B) = \sfrac{1}{8}$, so 
the shift-characterized solution again fails to partition $A$.
This is shown in \cref{f:auren1-32}.
\item
However, if $\nu(\vc{y}_1) = \sfrac{1}{4}$, then $\nu(\vc{y}_2) = 
\sfrac{3}{4}$. In this case 
$\mu(B) = 0$ and the shift-characterized solution does partition $A$.
See \cref{f:auren1-4}.
\end{enumerate}
Once again, \cref{e:suffUniqueMonge} fails in all the \cref{f:auren} cases:
\begin{equation*}
\mu\left( \setc*{\vc{x} \in A}
{c(\vc{x},\,\vc{y}_2) - c(\vc{x},\,\vc{y}_1) = k } \right) = 
\tfrac{1}{8}
\quad\quad
\text{ if }\quad\quad
k \in \set*{-\tfrac{\sqrt{2}}{2},\,0,\,\tfrac{\sqrt{2}}{2}}.
\end{equation*}
In fact, for this choice of $X$, $Y$, $\mu$, and $c$, the 
shift-characterized solution partitions $A$ if and only if
\begin{equation*}
\nu(\vc{y}_1) \in \left(\tfrac{1}{16},\,\tfrac{7}{16}\right)
\cup \left(\tfrac{9}{16},\,\tfrac{15}{16}\right).
\end{equation*}

Thus, as \cref{f:auren} illustrates, when $c$ is the 1-norm, one can have 
$\mu(B)=0$, giving a shift-characterized partition of $A$ that is unique 
$\mu$-a.e., whether or not \cref{e:suffUniqueMonge} is satsified.

\Cref{f:auren1-2} is worth special consideration, because it is not simply a 
non-partitioning shift-characterized transport solution:
it also constitutes a failed Voronoi diagram.
One can see a similar example in Figure 37 
of~\cite{Aurenhammer1991a}, offered as part of a discussion on methods for 
resolving lack of partitioning and uniqueness for certain Voronoi diagrams.

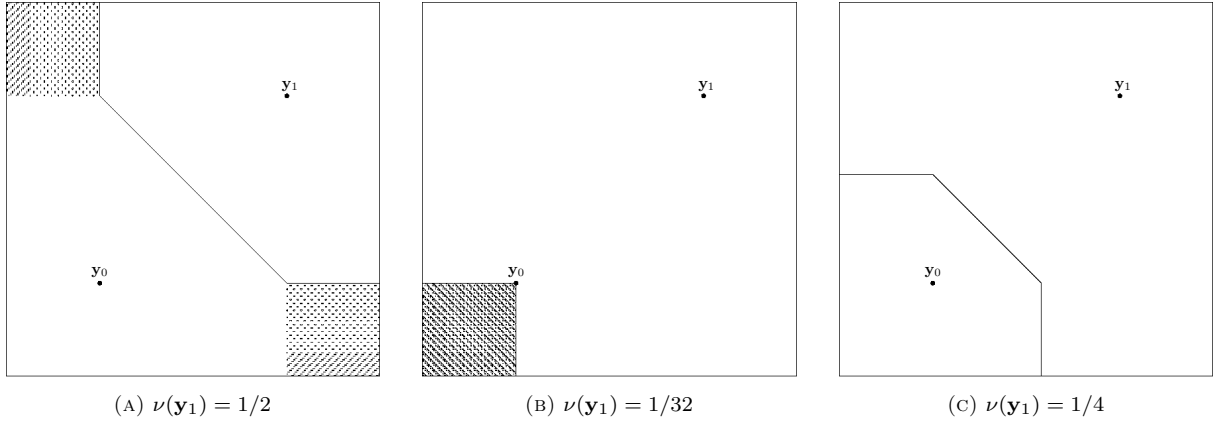
\begin{figure}[htpb]
  \centering
  \subfloat[$\nu(\vc{y}_1) = 1/2$]{%
    \label{f:auren1-2}
    \centering
    \resizebox*{0.3\textwidth}{!}{%
    \begin{overpic}[width=\textwidth]{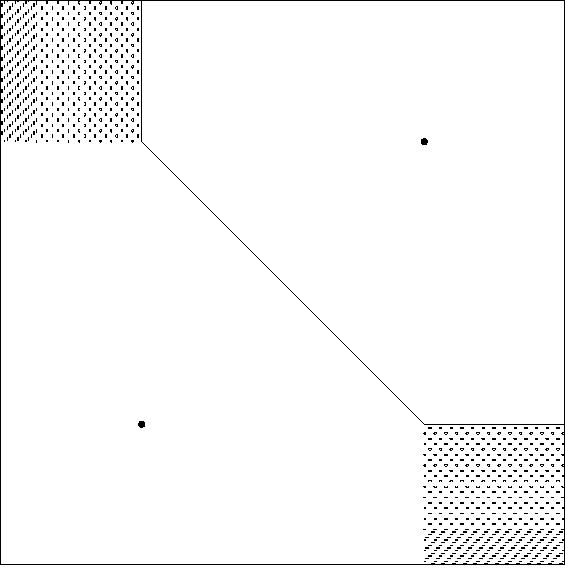}
    \put (22.5,28.5) {\scalebox{2.0}{$\vc{y}_0$}}
    \put (72.5,77.5) {\scalebox{2.0}{$\vc{y}_1$}}
    \end{overpic}}
  }%
  \quad%
  \subfloat[$\nu(\vc{y}_1) = 1/32$]{%
    \label{f:auren1-32}
    \centering
    \resizebox*{0.3\textwidth}{!}{%
    \begin{overpic}[width=\textwidth]{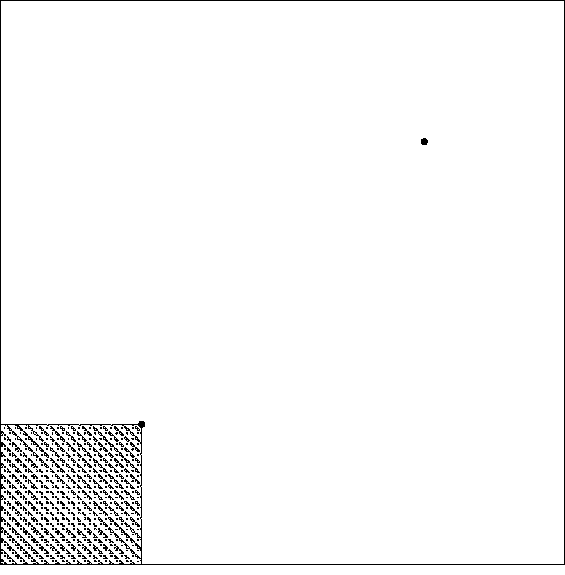}
    \put (22.5,28.5) {\scalebox{2.0}{$\vc{y}_0$}}
    \put (72.5,77.5) {\scalebox{2.0}{$\vc{y}_1$}}
    \end{overpic}}
  }%
  \quad%
  \subfloat[$\nu(\vc{y}_1) = 1/4$]{%
    \label{f:auren1-4}
    \centering
    \resizebox{0.3\textwidth}{!}{%
    \begin{overpic}[width=\textwidth]{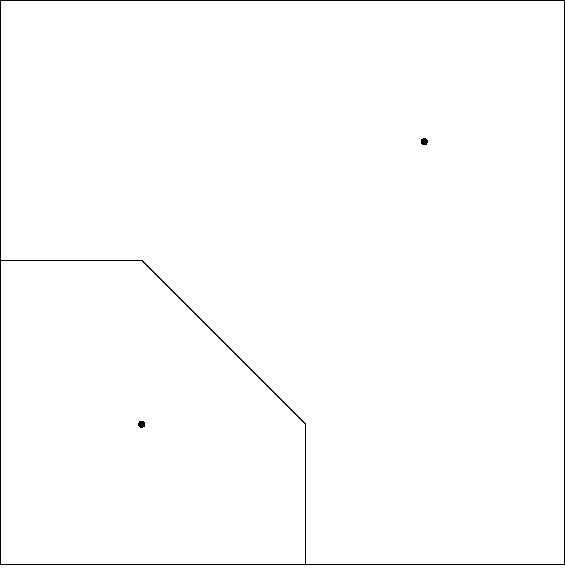}
    \put (22.5,28.5) {\scalebox{2.0}{$\vc{y}_0$}}
    \put (72.5,77.5) {\scalebox{2.0}{$\vc{y}_1$}}
    \end{overpic}}
  }
\caption{1-norm partitioning example}\label{f:auren}
\end{figure}

\subsection{Partitioning with \texorpdfstring{$\bm{p}$}{p}-norms when
\texorpdfstring{$\bm{p \in (1,\,\infty)}$}{p ∈ (1,∞)}-norm}\label{s:OtherPnorms}
Given the semi-discrete transport assumptions already stated,
let $c : \R^d \times \R^d \to \R$ be a $p$-norm with $p \in (1,\,\infty)$:
\begin{equation}
c(\vc{x},\,\vc{y}) :=
\left[\sum_{i=1}^d \abs{x_i - y_i}^p \right]^{1/p}\!\!\!\!\!\!\!\!\!,
\quad\quad
\forall\,
\vc{x} = (x_1,\,\ldots,\,x_d) \in X,
\quad\quad
\forall\,
\vc{y} = (y_1,\,\ldots,\,y_d) \in Y.
\end{equation}
Then the semi-discrete transport problem is always Monge under the 
shift characterization.

This assertion will be shown in two steps:
\begin{algolist}
\item\label{l:aijEqCost}
If $g_{ij}$, defined in \cref{gij}, is equal to the 
constant value $a_i - a_j$ in some 
neighborhood of $\vc{x}_0 \in A_{ij}$, then $\abs{a_i - a_j} = 
c(\vc{y}_i,\,\vc{y}_j)$.
[\cref{t:dirIncrease}]
\item\label{l:AllCollinear}
It follows from \cref{l:aijEqCost} that $\mu(B) > 0$
implies the existence of a 
ball of positive radius whose points are all collinear with both 
$\vc{y}_i$ and $\vc{y}_j$.
[\cref{t:WhenLpIsMonge}]
\end{algolist}
Because of the contradiction inherent in \cref{l:AllCollinear}, $\mu(B)=0$, and 
so \cref{t:WhenLpIsMonge} concludes that the problem must be Monge 
under the shift characterization.

\begin{theorem}\label{t:dirIncrease}
Let $c$ be a $p$-norm with $p \in (1,\,\infty)$, and $\vc{x}_0 \in 
A_{ij}$ for some $i,\,j \in \N_n$, $i \neq j$.
If $g_{ij}(\vc{x}) = a_i - a_j$ for all $\vc{x}$ in a neighborhood of 
$\vc{x}_0$, then $\abs{a_i - a_j} = c(\vc{y}_i,\,\vc{y}_j)$.
\end{theorem}
\begin{proof}
Let $c$ be a $p$-norm with $p \in (1,\,\infty)$, $\vc{x}_0 \in 
A_{ij}$, and $g_{ij}(\vc{x}) = a_i - a_j$ for all $\vc{x}$ in some neighborhood 
of $\vc{x}_0$.
Suppose to the contrary, however, that $\abs{a_i - a_j} \neq 
c(\vc{y}_i,\,\vc{y}_j)$.

Say $\abs{a_i - a_j} > c(\vc{y}_i,\,\vc{y}_j)$, and assume without loss of 
generality that $\abs{a_i - a_j} = a_i - a_j$. Then
\begin{align*}
g_{ij}(\vc{x}_0) = c(\vc{x}_0\,\vc{y}_i) - c(\vc{x}_0\,\vc{y}_j)
= a_i - a_j > c(\vc{y}_i,\,\vc{y}_j),
\end{align*}
which implies $c(\vc{x}_0\,\vc{y}_i) > c(\vc{x}_0\,\vc{y}_j)
+ c(\vc{y}_i,\,\vc{y}_j)$.
This is a violation of the triangle inequality.
Therefore, it must be the case that $\abs{a_i - a_j} < c(\vc{y}_i,\,\vc{y}_j)$.

For all $k \in \N_n$, define
$c_k(\vc{x}) := c(\vc{x},\,\vc{y}_k)$.
Because $\abs{a_i-a_j} < c(\vc{y}_i,\,\vc{y}_j)$, $\vc{x}_0 \neq \vc{y}_i$ and 
$\vc{x}_0 \neq \vc{y}_j$.
Hence, $c_i(\vc{x}_0) > 0$ and $c_j(\vc{x}_0) > 0$.

Because $g_{ij}$ is constant in a neighborhood of $\vc{x}_0$,
$\nabla g_{ij}(\vc{x}_0) =
\nabla c_i(\vc{x}_0) - \nabla c_j(\vc{x}_0) = 0$,
which implies $\nabla c_i(\vc{x}_0) = \nabla c_j(\vc{x}_0)$.
Hence, each of the first-order partial derivatives of $c_i$ and $c_j$ are equal 
at $\vc{x}_0$.

Assume $\vc{x}_0 = (x_1,\,\ldots,\,x_d)$,
$\vc{y}_i = (y^i_1,\,\ldots,\,y^i_d)$, and
$\vc{y}_j = (y^j_1,\,\ldots,\,y^j_d)$.
Then the equality of the $k$-th partial derivatives,
$\nabla_{\!x_k} c_i(\vc{x}_0) = \nabla_{\!x_k} c_j(\vc{x}_0)$,
gives
\begin{equation*}
(x_k - y^i_k)\abs{x_k - y^i_k}^{p-2}\left(c_i(\vc{x}_0)\right)^{1-p}
=
(x_k - y^j_k)\abs{x_k - y^j_k}^{p-2}\left(c_j(\vc{x}_0)\right)^{1-p}.
\end{equation*}
Thus, $x_k - y^i_k$ and $x_k - y^j_k$ have the same sign or are both zero.
Because $p > 1$, $p-1 > 0$.
Hence, taking the $(p-1)$-th root of both sides,
\begin{equation}\label{e:eqKthDerivSimp}
\frac{x_k - y^i_k}{c_i(\vc{x}_0)}
=
\frac{x_k - y^j_k}{c_j(\vc{x}_0)}
\quad\quad\quad\forall\,
k \in \N_d.
\end{equation}
As a consequence of \cref{e:eqKthDerivSimp}, $x_k - y^i_k = 0$ if and only if 
$x_k - y^j_k = 0$. Hence, $x_k = y^i_k$ if and only if $x_k = y^j_k$.

Let $K$ be the total number of $k$-th directional components satisfying $x_k 
\neq y^i_k$.
Consider three cases: $K = 0$, $K = 1$, and $K > 1$.
\begin{description}
\item[\textbf{$\bm{K=0}$}]
Then $x_k = y^i_k = y^j_k$ for all $k \in \N_d$, 
in which case $\vc{y}_i = \vc{y}_j$.
Since the semi-discrete transport problem requires distinct non-zero points in 
$Y$, it must be the case that $i = j$, contradicting the initial assumption 
that $i \neq j$.
Hence, $K \geq 1$.
\item[\textbf{$\bm{K = 1}$}]
There exists exactly one $k$ such that the components are not 
equal.
Since $x_k - y^i_k$ 
and $x_k - y^j_k$ have the same sign,
\begin{align*}
\abs{g_{ij}(\vc{x}_0)} = \abs{ (x_k - y^i_k) - (x_k - y^j_k) }
= \abs{ y^j_k - y^i_k } = c(\vc{y}_i,\,\vc{y}_j).
\end{align*}
This contradicts the assumption that $\abs{a_i - a_j} < 
c(\vc{y}_i,\,\vc{y}_j)$, and hence $K > 1$.
\item[$\bm{K>1}$]
Because $g_{ij}$ is constant in some neighborhood of $\vc{x}_0$, it must also 
be 
the case that $\nabla^2 g_{ij}(\vc{x}_0)=0$.
Hence, $\nabla^2 c_i(\vc{x}_0) = \nabla^2 c_j(\vc{x}_0)$, so each of the 
second-order partial derivatives of $c_i$ and $c_j$ are equal at $\vc{x}_0$.
The equality of the second-order partial derivatives taken with respect to 
$x_k$ gives
\begin{equation}
\frac{(p-1)\abs{x_k - y^i_k}^{p-2}}{(c_i(\vc{x}_0))^{2p-1}}
\left[ (c_i(\vc{x}_0))^p - \abs{x_k - y^i_k}^p \right]
=
\frac{(p-1)\abs{x_k - y^j_k}^{p-2}}{(c_j(\vc{x}_0))^{2p-1}}
\left[ (c_j(\vc{x}_0))^p - \abs{x_k - y^j_k}^p \right],
\end{equation}
which can be rewritten as
\begin{equation}\label{e:eqKthSecndSimp}
\frac{p-1}{c_i(\vc{x}_0)}
\left(\frac{\abs{x_k - y^i_k}}{c_i(\vc{x}_0)}\right)^{p-2}
\left[ 1 - \left(\frac{\abs{x_k - y^i_k}}{c_i(\vc{x}_0)}\right)^{p}
\right]
=
\frac{p-1}{c_j(\vc{x}_0)}
\left(\frac{\abs{x_k - y^j_k}}{c_j(\vc{x}_0)}\right)^{p-2}
\left[ 1 - \left(\frac{\abs{x_k - y^j_k}}{c_j(\vc{x}_0)}\right)^{p}
\right].
\end{equation}
Applying \cref{e:eqKthDerivSimp}, define
\begin{equation*}
\sigma_k = 
\frac{\abs{x_k - y^i_k}}{c_i(\vc{x}_0)}
=
\frac{\abs{x_k - y^j_k}}{c_j(\vc{x}_0)}.
\end{equation*}
Then \cref{e:eqKthSecndSimp} can be rewritten as
\begin{equation}\label{e:eqKthSecndSimp2}
\frac{p-1}{c_i(\vc{x}_0)}\, \sigma_k^{p-2}
( 1 - \sigma_k^{p} )
\\ =
\frac{p-1}{c_j(\vc{x}_0)}\, \sigma_k^{p-2}
( 1 - \sigma_k^{p} ).
\end{equation}

By assumption, for all $k \in \N_d$, $x_k - y^i_k \neq 0$ and 
$x_k - y^j_k \neq 0$.
Hence, $\sigma_k > 0$.

Since $d > 1$, and $\abs{x_k - y^i_k} > 0$ for all $k \in \N_d$, it must be 
that $\abs{x_k - y^i_k} < c_i(\vc{x}_0)$ for all $k \in \N_d$.
Therefore,
\begin{equation*}
\sigma_k = \frac{\abs{x_k - y^i_k}}{c_i(\vc{x}_0)} < 1,
\end{equation*}
which implies
$1 - \sigma_k^{p} > 0$.
Therefore,
$(p-1)\,\sigma_k^{p-2}( 1 - \sigma_k^{p} ) > 0$, and
\cref{e:eqKthSecndSimp2} simplifies to
$\frac{1}{c_i(\vc{x}_0)}
 =
\frac{1}{c_j(\vc{x}_0)}$.

Thus, $c_i(\vc{x}_0) = c_j(\vc{x}_0)$.
Combining this with \cref{e:eqKthDerivSimp} implies $y^i_k = y^j_k$ for all $k 
\in \N_d$, and so $\vc{y}_i = \vc{y}_j$.
Since $\vc{y}_i = \vc{y}_j$, and the semi-discrete transport problem requires 
distinct non-zero points in $Y$, it must be the case that $i=j$, contradicting 
the initial assumption that $i \neq j$.
Thus, $K \ngtr 1$.
\end{description}
All choices of $K$ lead to contradictions.
Hence, if $c$ is a $p$-norm for some $p \in (1,\,\infty)$, $\vc{x}_0 \in 
A_{ij}$, 
$i \neq j$, and $g_{ij}(\vc{x}) = a_i - a_j$ for all $\vc{x}$ in some 
neighborhood of $\vc{x}_0 \in A_{ij}$, then it must be the case that $\abs{a_i 
- a_j} = c(\vc{y}_i,\,\vc{y}_j)$.
\end{proof}

\begin{theorem}\label{t:WhenLpIsMonge}
If $c$ is a $p$-norm for some $p \in (1,\,\infty)$, then the semi-discrete 
transport problem is Monge under the shift characterization.
\end{theorem}
\begin{proof}
Assume the contrary is true.
Then $\mu(B) > 0$, so $\mu(A_{ij}) > 0$ for some $i,\,j \in \N_n$, $i \neq j$. 
Because $\mu$ is nonatomic, there exist $\vc{x}_0 \in A_{ij}$ and 
$\epsilon > 0$ such that the ball
$\mathcal{B}_{\epsilon}(\vc{x}_0)$, defined with respect to the Euclidean 
space $\R^d$, satisfies 
$\mathcal{B}_{\epsilon}(\vc{x}_0) \subseteq A_{ij}$ and
$\mu(\mathcal{B}_{\epsilon}(\vc{x}_0)) > 0$.
By \cref{t:dirIncrease}, $\abs{a_i - a_j} = c(\vc{y}_i,\,\vc{y}_j)$.
Assume without loss of generality that $\abs{a_i - a_j} = a_i - a_j$.

Let $\vc{x} \in \mathcal{B}_{\epsilon}(\vc{x}_0)$.
Since $\vc{x} \in A_{ij}$,
\begin{align*}
g_{ij}(\vc{x}) = a_i - a_j
\quad \iff \quad
c(\vc{x},\,\vc{y}_i) - c(\vc{x},\,\vc{y}_j) = c(\vc{y}_i,\,\vc{y}_j)
\quad \iff \quad
c(\vc{x},\,\vc{y}_i) = c(\vc{x},\,\vc{y}_j) + c(\vc{y}_i,\,\vc{y}_j).
\end{align*}
Because $c$ is a $p$-norm and $p \in (1,\,\infty)$, Minkowski's inequality 
implies 
that $\vc{x}$, $\vc{y}_i$, and $\vc{y}_j$ are all collinear.
The choice of $\vc{x}$ was nonspecific, and therefore every point in the ball 
$\mathcal{B}_{\epsilon}(\vc{x}_0)$ must be collinear with the points 
$\vc{y}_i$ and $\vc{y}_j$.

Of course, this is impossible, and so $\mu(A_{ij}) = 0$ for all $i,\,j \in 
\N_n$, $i \neq j$. Therefore, $\mu(B) = 0$.
From this final contradiction, it is clear that the semi-discrete transport 
problem must be Monge under the shift characterization.
\end{proof}

\begin{corollary}
If the semi-discrete transport problem is defined as given in 
\cref{s:semi-discrete}, and $c$ is a $p$-norm for some $p \in (1,\,\infty)$, 
then \cref{e:suffUniqueMonge} is satisfied, and the optimal transport solution 
is unique $\mu$-a.e.
\end{corollary}
\begin{proof}
Suppose $c$ is a $p$-norm for some $p \in (1,\,\infty)$, and assume the 
semi-discrete transport problem is characterized by shifts as given in 
\cref{ShiftChar}.
By the triangle inequality, for all $\vc{x} \in X$, $i,\,j \in \N_n$,
$i \neq j$,
$c(\vc{x},\,\vc{y}_i) \leq c(\vc{x},\,\vc{y}_j) + c(\vc{y}_i,\,\vc{y}_j)$.
Hence, one consequence of the triangle inequality is that 
$g_{ij}(\vc{x}) \leq c(\vc{y}_i,\,\vc{y}_j)$ for all $\vc{x} \in A$,
$i,\,j \in \N_n$ such that $i \neq j$.
Therefore,
$\setc{\vc{x} \in A}{g_{ij}(\vc{x}) = k } = \varnothing$
if $k < -c(\vc{y}_i,\,\vc{y}_j)$ or
$k > -c(\vc{y}_i,\,\vc{y}_j)$.
This implies
\begin{equation*}
\mu\left( \setc{\vc{x} \in A}
{g_{ij}(\vc{x}) = k } \right) = 0
\quad\quad
\forall\, i,\,j \in \N_n,\,i\neq j,
\quad\quad
\forall\, k \in (-\infty,\,-c(\vc{y}_i,\,\vc{y}_j)) 
\cup (c(\vc{y}_i,\,\vc{y}_j),\,\infty).
\end{equation*}

By \cref{t:WhenLpIsMonge}, $\mu(B) = 0$. Thus, for any $i,\,j \in \N_n$, $i 
\neq j$, $\mu(A_{ij}) = 0$.
Since the problem assumes nothing about the probability density $\nu$, it 
must be the case that
\begin{equation*}
\mu\left( \setc{\vc{x} \in A}
{g_{ij}(\vc{x}) = k } \right) = 0
\quad\quad
\forall\, i,\,j \in \N_n,\,i\neq j,
\quad\quad
\forall\, k \in [-c(\vc{y}_i,\,\vc{y}_j),\,c(\vc{y}_i,\,\vc{y}_j)].
\end{equation*}
Therefore,
$\mu\left( \setc{\vc{x} \in A}
{g_{ij}(\vc{x}) = k } \right) = 0$
for all $i,\,j \in \N_n$, $i\neq j$, and for all
$k \in \R$,
and uniqueness follows from Corollary 4 of~\cite{Cuesta1993a}.
\end{proof}

\section{Conclusions}\label{s:concl}
This paper resolves issues of partitioning and uniqueness for 
semi-discrete transport problems using a large class of ground cost functions: 
the $p$-norms.
If the cost function is a $p$-norm with $p \in (1,\,\infty)$, the above 
arguments ensure that $\mu$-a.e.\ unique solutions exist for 
semi-discrete transport problems.
As the examples show, if the cost function is a $p$-norm with $p=1$ or 
$p=\infty$, the solution may or may not constitute a $\mu$-a.e.\ unique 
partition of the continuous space.

\bibliographystyle{plain}
\bibliography{ref.bib}
\end{document}